\documentclass[12pt,leqno,a4paper]{amsart}
\usepackage{amssymb,enumerate}
\newcommand{\FF}{{\mathbb{F}}}

\newcommand{\RR}{{\mathbb{R}}}
\newcommand{\ZZ}{{\mathbb{Z}}}

\newcommand{\fA}{{\mathfrak{A}}}
\newcommand{\fS}{{\mathfrak{S}}}

\newcommand{\bB}{{\mathbf{B}}}
\newcommand{\bG}{{\mathbf{G}}}
\newcommand{\bL}{{\mathbf{L}}}
\newcommand{\bT}{{\mathbf{T}}}
\newcommand{\bU}{{\mathbf{U}}}
\newcommand{\bX}{{\mathbf{X}}}

\newcommand{\cE}{{\mathcal{E}}}

\newcommand{\Id}{{\operatorname{Id}}}
\newcommand{\Res}{{\operatorname{Res}}}
\newcommand{\Irr}{{\operatorname{Irr}}}
\newcommand{\mh}{{\operatorname{mh}}}
\newcommand{\hgt}{{\operatorname{ht}}}

\newcommand{\GL}{{\operatorname{GL}}}
\newcommand{\PGL}{{\operatorname{PGL}}}
\newcommand{\PSL}{{\operatorname{L}}}
\newcommand{\SL}{{\operatorname{SL}}}
\newcommand{\Sp}{{\operatorname{Sp}}}
\newcommand{\GO}{{\operatorname{GO}}}
\newcommand{\OO}{{\operatorname{O}}}
\newcommand{\GU}{{\operatorname{GU}}}
\newcommand{\SU}{{\operatorname{SU}}}
\newcommand{\PGU}{{\operatorname{PGU}}}
\newcommand{\PSU}{{\operatorname{U}}}
\newcommand{\Spin}{{\operatorname{Spin}}}
\newcommand{\wal}{{\widehat{\alpha}}}

\newcommand{\tw}[1]{{}^#1\!}
\def\pmod#1{~({\rm mod}~#1)}
\newcommand{\Ph}[1]{\Phi_#1}

\let\al=\alpha
\let\bt=\beta

\let\eps=\epsilon
\let\veps=\varepsilon

\newtheorem{thm}{Theorem}[section]
\newtheorem{lem}[thm]{Lemma}

\newtheorem{prop}[thm]{Proposition}

\theoremstyle{definition}

\newtheorem*{conjA}{Conjecture}

\theoremstyle{remark}
\newtheorem{rem}[thm]{Remark}

\begin{document}

\title[Characters of positive height]{Characters of positive height\\[4pt] in blocks of finite quasi-simple groups}

\date{\today}

\author{Olivier Brunat \and Gunter Malle}
\address{Universit\'e Paris-Diderot Paris 7, Institut de math\'ematiques de
         Jussieu, UFR de math\'e\-matiques, Case 7012, 75205 Paris Cedex 13,
         France.}
\email{brunat@math.jussieu.fr}
\address{FB Mathematik, TU Kaiserslautern, Postfach 3049,
         67653 Kaisers\-lautern, Germany.}
\email{malle@mathematik.uni-kl.de}

\thanks{The second author gratefully acknowledges financial support by ERC
  Advanced Grant 291512.}


\begin{abstract}
Eaton and Moret\'o proposed an extension of Brauer's famous height zero
conjecture on blocks of finite groups to the case of non-abelian defect
groups, which predicts the smallest non-zero height in such blocks in
terms of local data. We show that their conjecture holds for principal blocks
of quasi-simple groups, for all blocks of finite reductive groups in their
defining characteristic, as well as for
all covering groups of symmetric and alternating groups. For the proof, we
determine the minimal non-trivial character degrees of Sylow $p$-subgroups
of finite reductive groups in characteristic~$p$. We provide some further
evidence for blocks of groups of Lie type considered in cross characteristic.
\end{abstract}

\maketitle


\section{Introduction}   \label{sec:intro}

Let $G$ be a finite group, $p$ a prime and $B$ a $p$-block of $G$ with defect
group $D$. The famous Height Zero Conjecture of Richard Brauer from 1955 states
that all irreducible complex characters in $B$ have the same $p$-part in their
degree if and only if $D$ is abelian. Recently, substantial progress on this
conjecture has been made: the 'if' direction was proved by Kessar--Malle
\cite{KM13}, and the other direction was reduced to the inductive Alperin--McKay
condition for quasi-simple groups by Navarro--Sp\"ath \cite{NS13}.
Eaton and Moret\'o \cite{EM14} have recently proposed an extension of this
conjecture to blocks with non-abelian defect groups. Write $\mh(B)$ for the
minimal non-zero height of an irreducible character in $B$, and similarly (by a
slight abuse of notation) $\mh(D)$ for the minimal non-zero height of any
irreducible character of $D$ (and $\mh=\infty$ if there is no character of
positive height).

\begin{conjA}[Eaton--Moret\'o]
 \emph{Let $B$ be a $p$-block of a finite group with defect group $D$.
 Then $\mh(B)=\mh(D)$.}
\end{conjA}

Brauer's height zero conjecture is included as the claim that $\mh(B)=\infty$
if and only if $\mh(D)=\infty$. Eaton and Moret\'o \cite{EM14} prove that
$\mh(D)\le\mh(B)$ for $p$-solvable groups, and
they furthermore checked their conjecture for sporadic groups, for the
symmetric groups and for the general linear groups $\GL_n(q)$ for the
prime~$p$ dividing $q$.
\par
It is the purpose of this paper to verify the conjecture in a number of further
instances. Relevant test cases are certainly furnished by the blocks of
nearly simple groups. Our main result is the following:

\begin{thm}   \label{thm:main}
 The Eaton--Moret\'o conjecture holds for the following blocks:
 \begin{enumerate}[\rm(1)]
  \item the principal block of any quasi-simple group;
  \item all $p$-blocks of quasi-simple groups of Lie type in characteristic~$p$;
  \item all unipotent blocks of quasi-simple exceptional groups of Lie type ;
   and
  \item all $p$-blocks of covering groups of an alternating or symmetric group.
 \end{enumerate}
\end{thm}
This is shown in Theorems~\ref{thm:B0}, \ref{thm:Liep}, \ref{thm:exc}
and~\ref{thm:alt}
respectively. We also obtain further partial results for groups of Lie type in
non-defining characteristic, see Proposition~\ref{prop:smallexc}. The case of
groups of Lie type in their defining characteristic seems particularly
relevant for the conjecture since there $\mh(B)$ and $\mh(D)$ can take
arbitrarily large positive values. As a side result of independent interest
we determine
the smallest non-trivial character degrees of Sylow $p$-subgroups of groups
of Lie type in characteristic~$p$, see Proposition~\ref{prop:defgr}.

\section{Alternating groups} \label{sec:alt}

In this section we consider the quasi-simple coverings of alternating groups.

\begin{thm}   \label{thm:alt}
 Let $G$ be a symmetric or alternating group, or a Schur extension of one of
 these groups, and $p$ a prime. Then $\mh(B)=\mh(D)=1$ or
 $\mh(B)=\mh(D)=\infty$ for any $p$-block $B$ of $G$.
\end{thm}

\begin{proof}
The case of the symmetric group is proved in~\cite[Thm.~4.3]{EM14}. Let $B'$
be a $p$-block of $\fA_n$ of $p$-weight $w>0$, and $B$ be a $p$-block of
$\fS_n$ with defect group $D$ covering $B'$. By~\cite[Thm.~9.17]{Nav},
$D'=D\cap \fA_n$ is a defect group of $B'$. In particular, when $p$ is odd,
$D'=D$ and we obtain the result from the symmetric groups case using
\cite[Prop.~12.3]{Ol93}. Suppose $p=2$. Then $D'$ is a subgroup of
index $2$ in $D$ isomorphic to a Sylow $2$-subgroup of $\fA_{2w}$ by
\cite[Prop.~11.2]{Ol93}. If $w\leq 2$, then $D'$ is abelian, and
$\mh(D')=\infty=\mh(B')$. Suppose $w\geq 3$. By~\cite[Prop.~12.7]{Ol93},
we have $\mh(B')=1$. We now will show that $\mh(D')=1$. For $w=3$, $D'$ is
isomorphic to a Sylow $2$-subgroup of $\fA_6$, which has an irreducible
character of degree~$2$, as required. Let $w>3$. By Clifford theory, it is
sufficient to show that $D$ has an irreducible character $\chi$ of degree~$2$
such that $\varepsilon\chi\neq \chi$, where $\varepsilon$ is the restriction
of the sign character of $\fS_n$ to $D$. However, $w>3$ implies that $2w\geq 8$,
and if we write $2w=\sum_i a_i2^i$ with $a_i\in\{0,1\}$, then there is
$k\geq 3$ such that $a_k\neq 0$. In particular, $D$ (which is isomorphic to
a Sylow $2$-subgroup of $\fS_{2w}$ described, for example,
in~\cite[p.~75]{Ol93}) possesses a quotient isomorphic to
$(\ZZ/2\ZZ\wr\ZZ/2\ZZ)\wr\ZZ/2\ZZ$ which has an irreducible character of
degree $2$. The result follows.

Suppose $\widetilde{G}=\widetilde{\fS}_n$ is a Schur extension of $G=\fS_n$,
and denote by $\theta:\widetilde{G}\rightarrow G$ the surjective homomorphism
of groups with kernel $Z(\widetilde{G})\simeq\ZZ/2\ZZ$.  First, assume that
$p>2$. Let $w>0$ be the $p$-weight of $B$. By~\cite[Prop.~13.3]{Ol93},
$D$ is conjugate in $\widetilde{\fS}_n$ to a Sylow $p$-subgroup $P$ of
$\widetilde{\fS}_{pw}$ and since $p$ is odd, $\theta|_P$ induces an isomorphism
between $P$ and a Sylow $p$-subgroup of $\fS_{pw}$. In particular, by the
proof of~\cite[Thm.~4.3]{EM14}, either $D$ is abelian and $\mh(D)=\infty$,
or $D$ is non-abelian and $\mh(D)=1$. We then conclude
with~\cite[note p.~86]{Ol93}.  Now, let $B'$ be a $p$-block of
$\widetilde{\fA}_n$ of $p$-weight $w>0$ and sign $\sigma$ (where $\sigma$ is
the sign as in~\cite[p.~45]{Ol93} of the partition labelling the core of
$B'$), with non-abelian defect group $D$.  Then the unique $p$-block of
$\widetilde{\fS}_n$ covering $B'$ has defect group $D$. Thus, $\mh(D)=1$.
Let $B$ be any $p$-block of $\widetilde{\fS}_m$ (for some positive integer $m$)
of weight $w$ and sign $-\sigma$. Then, by~\cite[Prop.~13.19]{Ol93},
we conclude that $B$ and $B'$ have the same number of characters of
$p$-height $1$. So, $\mh(B')=1$, as required.

Finally, suppose that $p=2$. Write $\widetilde{G}=\widetilde{\fS}_n$ and
$G=\fS_n$ (resp.~$\widetilde{G}=\widetilde{\fA}_n$, and $G=\fA_n$).
Since $Z(\widetilde{G})$ is a $2$-group, \cite[Thm.~9.9(b)]{Nav} gives that
if $\widetilde{B}$ is a $2$-block of $\widetilde{G}$ with defect group
$\widetilde{D}$, then there is a $2$-block $B$ of $G$ with defect group
$D=\theta(\widetilde{D})$. Assume $D$ is non-abelian. Inflating irreducible
characters of $B$ and $D$ of degree $2$ through $\theta$, we obtain
$\mh(\widetilde{B})=1=\mh(\widetilde{D})$.
So we now consider the case that $D$ is abelian. We remark that in this
case, $D$ is either trivial, or isomorphic to $\ZZ/2\ZZ$. In
particular, $\widetilde{D}$  is abelian, and
$\mh(\widetilde{D})=\infty=\mh(\widetilde{B})$.

To complete the proof, we now consider the cases of $6.\fA_6$ and $6.\fA_7$.
Let $G=6.\fA_6$, and $p\in\{2,3\}$. Then we can check that the $p$-blocks
of $G$ have either maximal defect (and have an irreducible character of
height $1$) or cyclic defect. Since a Sylow $p$-subgroup has an irreducible
character of height $1$, we deduce the result.

Let $G=6.\fA_7$. If $p=2$, then $G$ has three $2$-blocks with maximal defect,
one $2$-block with defect $8$, and three $2$-blocks with cyclic defect.
Any $2$-block with maximal defect and any Sylow $2$-subgroup of $G$ has an
irreducible character of height $1$. Furthermore, if $B$ is the $2$-block
with defect $8$, then $B$ has a character of degree $6$ (i.e., of height $1$)
and a defect group $D$ of $B$ is a non-abelian group of order $8$, which
has an irreducible character of degree $2$. Thus, $\mh(D)=1=\mh(B)$.
If $p=3$, then the $3$-blocks of $G$ are either of maximal defect (and have
characters of height $1$) or of abelian defect. Since a Sylow $3$-subgroup
of $G$ has a character of degree $3$, the result follows.

Finally, for $G=3.\fA_6$ and $G=3.\fA_7$ the $2$-blocks are either of maximal
defect (and contain a character of height $1$), or of abelian defect.
Furthermore, a Sylow $2$-subgroup of $G$ possesses a character of degree $2$.
This completes the proof.
\end{proof}

\section{Groups of Lie-type: defining characteristic} \label{sec:Liep}

\subsection{Character degrees}
Throughout, we consider the following setting: $\bG$ is a simple linear
algebraic group over the algebraic closure of a finite field,
and $F:\bG\rightarrow\bG$ is a Steinberg endomorphism on $\bG$, with finite
group of fixed points $G:=\bG^F$. We denote by $q$ the absolute value of all
eigenvalues of $F$ on the character group of an $F$-stable maximal torus of
$\bG$. Thus, if $\delta$ is the smallest integer such that $\bG$ is split
with respect to $F^\delta$, then $\bG$ is defined over the field of size
$q^\delta$. We write $p$ for the prime dividing $q^\delta$. It is well-know
that all quasi-simple finite groups of Lie type, except for the finitely many
exceptional covering groups and for the Tits simple group $\tw2F_4(2)'$, arise
as a central factor group of a suitable $G$ as above, by choosing $\bG$ to
be of simply connected type.

We will make use of the following well-known property of character degrees
of groups of Lie type, which essentially follows from the work of Lusztig:

\begin{prop}   \label{prop:degs}
 Let $G=\bG^F$ be a finite quasi-simple group of Lie type and $q$ as above.
 Assume that $G$ is not of Suzuki- or Ree-type. Then the degree of any
 $\chi\in\Irr(G)$ is of the form
 $$\frac{1}{cd}f(q)$$
 for some monic polynomial $f\in\ZZ[X]$, and positive integers $c,d$, with
 $(c,q)=1$, where $c$ is a divisor of $|Z(\bG)|$ and $d$ is only divisible by
 bad primes for $\bG$.
\end{prop}

A quite similar statement holds for the Suzuki and Ree groups, but since it
is slightly more technical to formulate, and we will not need it here, we
do not give it.

\begin{proof}
Let us first consider the unipotent characters of $G$. Here, Lusztig has given
explicit formulas for their degrees: for each unipotent character $\chi$ of
$G$ there is a polynomial $f\in\ZZ[X]$ which is product of cyclotomic
polynomials times a power $X^a$ of $X$ and a positive integer $d$ which is
either a power of~2 or a divisor of~120, divisible only by bad primes for
$\bG$, such that $\chi(1)=\frac{1}{d}f(q)$ (see e.g.~\cite[\S13]{Ca}). (Here,
$a$ is the $a$-invariant of the family to which $\chi$ belongs.)
So the claim holds in this case.  \par
In general, by Lusztig's parameterization, any irreducible character $\chi$ of
$G$ lies in the Lusztig-series $\cE(G,s)$ of a semisimple element $s\in G^*$,
where $G^*:={\bG^*}^F$ with $\bG^*$ Langland's dual to $\bG$ with a compatible
Steinberg endomorphism also denoted by $F$. Lusztig's Jordan decomposition
then yields that
$$\chi(1)=[G^*:C_{G^*}(s)]_{p'}\,\psi(1)$$
for a unipotent character $\psi$ of (the possibly disconnected group)
$C_{G^*}(s)$, that is, an irreducible character whose restriction to
the connected component $C_{G^*}^\circ(s)$ is unipotent. Now
$C_{G^*}(s)/C_{G^*}^\circ(s)$ is isomorphic to a subgroup of the fundamental
group of $G^*$, hence abelian, of order prime to $q$, isomorphic to a subgroup
of $Z(\bG)$ (see \cite[Prop.~14.20]{MT}). Moreover, it is cyclic unless
the quasi-simple group $G$ is of type $D_{2n}$. Since the index of the
connected subgroup $C_{G^*}^\circ(s)$ in $G^*$ is given by a monic polynomial
in $q$, this shows the claim using the previous statement on unipotent
character degrees and Clifford theory.
\end{proof}

\subsection{$\mh(B)$ in the defining characteristic}
We now determine the smallest positive height in blocks of groups of Lie
type. Note that groups of type $A_1$ have abelian Sylow $p$-subgroups, so
we may exclude them here. For $G=\bG^F$ a quasi-simple group of Lie type in
characteristic $p$ but not of type $A_1$, with $q$ as above, we define the
positive integer $m(G,p)$ by
$$p^{m(G,p)}:=\begin{cases}
   \frac{1}{2}q& \text{ if $G=B_n(q),C_n(q),F_4(q),G_2(q)$ with $q=2^f>2$},\\
   \frac{1}{3}q& \text{ if $G=G_2(q)$ with $q=3^f>3$},\\
   \frac{1}{\sqrt{2}}q& \text{ if $G=\tw2B_2(q^2),\tw2F_4(q^2)$ with $q^2=2^{2f+1}>2$},\\
   \frac{1}{\sqrt{3}}q& \text{ if $G={}^2G_2(q^2)$ with $q^2=3^{2f+1}>3$},\\
   q& \text{ else}.
 \end{cases}$$

\begin{prop}   \label{prop:Liep}
 Let $G=\bG^F$ be a quasi-simple group of Lie type not of type $A_1$.
 Let $B$ be a $p$-block of $G$ of positive defect. Then $\mh(B)=m(G,p)$.
\end{prop}

\begin{proof}
By a result of Humphreys \cite{Hum} the $p$-blocks of $G$ of non-zero defect
are in bijection with the irreducible characters of $Z(G)$, and all have full
defect.  \par
Thus, the Suzuki and Ree groups have just one non-trivial block, and the claim
for them can be checked from the know character tables \cite{Suz,Wa,Ma90}. Note
that none of the groups $\tw2B_2(2),{}^2G_2(3),\tw2F_4(2)$ is quasi-simple.
Also, it can be checked directly from the character tables that
$\Sp_6(2),G_2(3)$ and $F_4(2)$ have a character of height~1. So from now on
we will assume that $G$ is none of the aforementioned groups. \par
Let us first discuss the principal $p$-block $B_0$. Since unipotent characters
have $Z(G)$ in their kernel, they all lie in $B_0$. By
Proposition~\ref{prop:degs} the degree of any unipotent character $\chi$ of
$G$ is of the form $\frac{1}{d}q^a f(q)$, where $a$ is the $a$-invariant of
the family of the Weyl group $W$ of $G$ to which $\chi$ is attached, and $f$
is a product of cyclotomic polynomials in $q$.  \par
For $B_n(2)$ and $C_n(2)$ with $n>3$, the unipotent character parametrized by
the symbol $(0,2\mid n-1)$ has degree $2(4^n-1)(2^{n-1}+1)(2^{n-3}+1)/15$,
hence height~1.
For the other groups $G$ of split type, we take for $\chi$ the unipotent
principal series character of $G$ indexed by the reflection character $\rho$
of $W$. Since
$\rho$ occurs in the first exterior power of the reflection representation
of $W$, its $b$-invariant equals~1 by definition. By a result of Lusztig,
the $a$-invariant is always less or equal to the $b$-invariant. On the other
hand the $a$-invariant is strictly positive for any family except the one
containing the trivial character. (Alternatively, use that the reflection
character is special and so $a$- and $b$-invariant agree,see
e.g.~\cite[\S12]{Ca}.) Thus $\chi(1)=\frac{1}{d}qf(q)$, with $d$ divisible
by bad primes only and $f(q)$ prime to $p$.
\par
We claim that this character has the desired height. This is clear for groups
of type $A$, since there are no bad primes, and more generally if $p$ is not
a bad prime for $\bG$. The explicit formulas in \cite[\S13]{Ca} show that for
types $B_n$, $C_n$ and $F_4$ we have $d=2$, for $G_2$ we have $d=6$, and
$d=1$ for all other untwisted types, so the claim holds.
\par
For the twisted groups of types
$\tw2A_n$, $\tw2D_{2n+1}$ and $\tw2E_6$, the degrees of unipotent characters
are obtained from those in the corresponding untwisted type by replacing $q$
by $-q$ in the degree polynomial, and adjusting signs (the so-called Ennola
duality), so the claim follows. For type $\tw2D_n$ with $n$ even, the
unipotent character parametrized by the symbol $(1,n-1\mid\emptyset)$ has
degree $q(q^n+1)(q^{n-2}-1)/(q^2-1)$. Finally, for $\tw3D_4(q)$ the unipotent
character $\phi_{1,3}''$ has degree $q(q^4-q^2+1)$.
\par
We have thus in all cases exhibited a (unipotent) character of the asserted
height. We next claim that the given characters have minimal $p$-height among
unipotent characters. This is again clear for types $A_n$ and $\tw2A_n$ by
Proposition~\ref{prop:degs}. For the exceptional types it follows by inspection
from the lists in \cite[\S13]{Ca}. For $B_n,C_n,D_n$ and $\tw2D_n$ the
formulas in loc.~cit.~show that the $a$-value of a unipotent character is
always at least the exponent of~2 in the denominator of the degree. This
settles the case for unipotent characters.
\par
Now let $\chi$ be an arbitrary character of $G$ of positive height. By the
degree formula in Proposition~\ref{prop:degs} its height is the same as that
of some unipotent character of the centralizer of some semisimple element of
the dual group $G^*$. But unipotent characters of products are just obtained
as products of the unipotent characters of the factors. Moreover, any
centralizer of a semisimple element in a group of simply laced type is again
of simply laced type. Thus the previous result on unipotent characters shows
that in all types the non-zero heights of non-unipotent characters are not
smaller than those of unipotent characters. This completes the proof for the
principal block.
\par
Now let $\psi$ denote a non-trivial character of $Z(G)$ and $B_\psi$ the
corresponding $p$-block of $G$. Note that by the first part since $|Z(G)|>1$
we are in the case where $p^{\mh(B_0)}=q$. For each type we display in
Table~\ref{tab:ss} the centralizer of a semisimple element $s$ in the dual group
$G^*$ such that all elements in the Lusztig series $\cE(G,s)$ have central
character $\psi$, hence lie in $B_\psi$.

\begin{table}[htbp]
\caption{Semisimple elements}   \label{tab:ss}
\[\begin{array}{|r|r|l|}
\hline
 G& C_{G^*}(s)& \text{conditions}\\
\hline
   \SL_n(q)& \GL_{n-1}(q)& (n,q-1)>1\\
   \SU_n(q)& \GU_{n-1}(q)& (n,q+1)>1\\
 \Spin_{2n+1}(q)& C_{n-1}& q\text{ odd}\\
     \Sp_{2n}(q)& \GO_{2n}^\pm(q)& q\text{ odd}\\
 \Spin_{2n}^\pm(q)& B_{n-1}& q\text{ odd}\\
     E_6(q)& D_5& q\equiv1\pmod3\\
 \tw2E_6(q)& \tw2D_5& q\equiv2\pmod3\\
     E_7(q)&     E_6& q\equiv1\pmod4\\
     E_7(q)& \tw2E_6& q\equiv3\pmod4\\
\hline
\end{array}\]
\end{table}

Let $\chi$ be the character in $\cE(G,s)$ corresponding under Lusztig's Jordan
decomposition to the unipotent character $\chi'$ of $C_{G^*}(s)$ constructed
above, with $\chi'(1)_p=q$. But then we also have $\chi(1)_p=q$. The proof is
complete.
\end{proof}

\begin{rem}
It is shown in \cite[Thm.~C]{EM14} that under Dade's projective conjecture
we always have $\mh(B)\ge\mh(D)$. Thus, assuming that deep conjecture, the
part of the above proof in which we show that the given characters have
minimal $p$-height could be omitted in view of the subsequent
Proposition~\ref{prop:defgr}.
\end{rem}

\subsection{$\mh(D)$ in the defining characteristic}  \label{subsec:2.2}
We now turn to the heights in a Sylow $p$-subgroup.
Let $\bT$ be an $F$-stable maximal torus of $\bG$ contained in an $F$-stable
Borel subgroup $\bB$ of $\bG$ and let $\bU$ be the unipotent radical of $\bB$.
We denote by $\Phi$ the root system of $\bG$ with respect to $\bT$, and by
$\Phi^+$ the set of positive roots of $\Phi$ with respect to $\bB$.
Write $\Delta$ for the corresponding simple roots.
For any $\al\in\Phi^+$, we denote by $\bX_\al$ the corresponding root subgroup
in $\bU$ normalized by $\bT$, and we choose an isomorphism
$x_\al:\overline{\FF}_p^+\rightarrow \bX_\al$. Now $F$ acts on the subgroups
$\bX_\al$, which induces an action of $F$ on $\Phi$ and on $\Delta$, and that
we extend by linearity to the space $V=\RR\Phi$. The resulting map is $q\phi$,
where $\phi$ is an automorphism of $V$ of finite order. We recall the setup
from \cite[\S23]{MT}. Write
$\pi:V\rightarrow V^{\phi}$ for the projection onto the subspace $V^{\phi}$.
We define an equivalence relation $\sim$ on $\Sigma=\pi(\Phi)$ by setting
$\pi(\al)\sim\pi(\bt)$ if and only if there is some positive $c\in\RR$ such
that $\pi(\al)=c\pi(\bt)$, and we let $\widehat{\Phi}$ be the set of equivalence
classes under this relation. For $\al\in\Phi$, we write $\wal$ for the class
of $\pi(\al)$ and define
\begin{equation}   \label{eq:thetadef}
  \theta:\Phi^+\rightarrow \widehat{\Phi},\ \al\mapsto\wal.
\end{equation}
Recall that $\Phi_\wal=\theta^{-1}(\wal)$ is an $F$-stable set of positive
roots of the root system $\pm\Phi_\wal$ (see~\cite[Thm.~2.4.1]{GLS}). We
sometimes also write $\Phi_{\pi(\al)}$ for $\Phi_\wal$. From now on, we assume
that we made the same choices as in \cite[Rem.~1.12.10]{GLS}.
For $\al\in\Phi^+$,
we define $\bX_\wal=\prod_{\bt\in\Phi_\wal}\bX_\bt$, and $X_\wal=\bX^F_\wal$
the corresponding root subgroup of $G$. In the following, to simplify the
notation, we sometimes also denote $X_\wal$ by $X_{\pi(\al)}$. The different
possible roots subgroups are described in~\cite[Table~2.4]{GLS}. For any
$\al\in\Phi^+$, when $\Phi_\wal$ is of type $A_1$ (resp.~$A_1\times A_1$),
we label the elements of $X_\wal$ by $x_\wal(t)$ for $t\in\FF_q$
(resp.~$t\in\FF_{q^2}$). When $\Phi_\wal$ is of type $A_2$, we label the
elements of $X_\wal$ by $x_\wal(t,u)$ with $t,\,u\in\FF_{q^2}$ such that
$t^{q+1}=\eps(u+u^q)$, where $\eps$ is a sign depending only on
$\wal$. For any representative $\gamma$ of $\wal$, we also denote $x_\wal$
by $x_{\gamma}$. Furthermore, by~\cite[Thm.~2.3.7]{GLS},
$P=\prod_{\{\wal,\,\al\in\Phi^+\}}X_\wal$ is a Sylow $p$-subgroup of $G$.
\medskip

Now, we consider $\bG=\SL_3(\overline{\FF}_p)$, endowed with the
$\FF_{q^2}$-structure corresponding to the Steinberg map $F$ acting by raising
all entries of a matrix to the $q^2$-th power. Let $\{\al_1,\al_2\}$ be a
system of simple roots of $\bG$. Then $\Phi^+=\{\al,\,\bt,\,\al +\bt\}$,
and by~\cite[Thm.~2.4.5(b)(1)]{GLS}, we have
$[x_\al(u),x_\bt(v)]=x_{\al+\bt}(\eps uv)$ for all $u,\,v\in\FF_{q^2}$,
where $\eps$ is a sign independent of $u$ and $v$.
Define $Y_\bt=x_\bt(\FF_q)$. In the following, we will need to understand
the characters degrees of the group $Y=X_\al Y_\bt X_{\al+\bt}$ of order $q^5$.

\begin{lem}   \label{lem:Aodd}
 The group $Y$ has $q^3$ linear characters and $q^3-q$ irreducible characters
 of degree $q$.
\end{lem}

\begin{proof}
Using the commutator relations, every elements of $Y$ can be uniquely written
as $x_\al(u)x_\bt(v)x_{\al+\bt}(w)$ for $u,\,w\in\FF_{q^2}$ and $v\in \FF_q$.
Thus $|Y|=q^5$. Now, we describe the conjugacy classes of $Y$. Using the
commutator relations, we can show that if $v\neq 0$, then for every
$u\in\FF_{q^2}$, the class of $x_\al(u)x_\bt(v)$ is
$\{x_\al(u)x_\bt(v)x_{\al+\bt}(w)\mid w\in \FF_{q^2}\}$, which has $q^2$
elements. There are $q^2(q-1)$ such classes. Let $u\in\FF_{q^2}^{\times}$.
Then the set $\{x_\al(u)x_{\al+\bt}(w)\,|\,w\in\FF_{q^2}\}$ is the union of
$q$ classes of $Y$ of size $q$. Finally, there are $q^2$ central classes with
representative in $X_{\al+\bt}$. Hence, $Y$ has $2q^3-q$ classes and
$|\Irr(Y)|=2q^3-q$. Since $[Y,Y]=X_{\al+\bt}$, we deduce that $Y$ has $q^3$
linear characters, and so
\begin{equation}
  \sum_{\chi\in\Irr(Y),\,\chi(1)\neq 1}\chi(1)^2=q^2(q^3-q).  \label{eq:rel1}
\end{equation}
Moreover, the subgroup $X_\al X_{\al+\bt}$ is abelian of index~$q$.
Thus, for any $\chi\in\Irr(Y)$, one has $\chi(1)\leq q$, and we derive the
result from Relation~(\ref{eq:rel1}).
\end{proof}

\begin{lem}   \label{lem:rktwo}
 Let $G=\bG^F$ be a quasi-simple group of Lie type, with $q$ as above. Assume
 that $p\notin\{2,3\}$. Let $\al',\,\bt'\in\Phi^+$ be such that $\al=\pi(\al')$
 and $\bt=\pi(\bt')$ are linearly independent, and consider
 $Q:=\langle X_{\widehat{\al'}},X_{\widehat{\bt'}}\rangle$.
 If $\chi$ is a character of $Q$ with $\chi(1)<q$, then $[Q,Q]\leq \ker(\chi)$.
\end{lem}

\begin{proof}
Let $\chi$ be a character of $Q$ such that $\chi(1)<q$. Denote by $R$ the root
subsystem generated by $\al$ and $\bt$. If $R$ is of type $A_1\times A_1$ then
$Q$ is abelian and the claim is obvious, so we may exclude that case from
now on.

First, we consider the case of untwisted groups. In this case, $F$ acts
trivially on $\Phi$, and $\pi=\Id$. Suppose that $R$ is of type $A_2$.
By~\cite[Thm.~1.12.1]{GLS}, $Q=X_\al X_\bt X_{\al+\bt}$ and
$[x_\al(u),x_\bt(v)]=x_{\al+\bt}(\eps uv)$ for all $u,\,v\in\FF_q$,
where $\eps$ is a sign depending only on $\al$ and $\bt$. Thus,
$Z(Q)=[Q,Q]=X_{\al+\bt}$. Furthermore, for any $g\notin X_{\al+\bt}$ and
$c\in X_{\al+\bt}$, we derive from the commutator relation that there is
$t\in Q$ such that $[g,t]=c$. Hence, by~\cite[Thm.~7.5]{Hup}, the degrees of
the irreducible characters of $Q$ are $1$ or $q$. Since $\chi(1)<q$, it
follows that $\chi$ is a sum of linear characters. In particular, $[Q,Q]$ lies
in the kernel of $\chi$, as required. In the following, we denote the group
$Q$ considered here by $Q_{A_2}$.

Assume that $R$ is of type $B_2$, and $R^+=\{\al,\,\bt,\,\al+\bt,\,2\al+\bt\}$.
By the commutator relations~\cite[Thm.~1.12.1]{GLS},
$Y=X_\al X_{\al+\bt}X_{2\al+\bt}$ is a subgroup of $Q$ isomorphic to $Q_{A_2}$,
because $p\neq 2$; an isomorphism from $Q_{A_2}$ to $Y$ is given by
$x_\al(u)x_\bt(v)x_{\al+\bt}(w)\mapsto x_\al(u)x_{\al+\bt}(v)x_{2\al+\bt}(2w)$.
In particular, the argument for type $A_2$ can be applied to
$\Res_Y^Q(\chi)$. This implies that $X_{2\al+\bt}\leq\ker(\chi)$.
Since $X_{2\al+\bt}\lhd Q$, $\chi$ factorizes through the quotient
$\overline{Q}=Q/X_{2\al+\bt}$. By the commutator relations,
$\overline{Q}=\overline{X}_\al\overline{X}_\bt\overline{X}_{\al+\bt}$,
where $\overline{X}_{\gamma}=\{\overline{x}_{\gamma}(t)\,|\,t\in\FF_{q}\}$
is isomorphic to $X_{\gamma}$ for $\gamma\in\{\al,\bt,\al+\bt\}$, and satisfies
$[\overline{x}_\al(u),\overline{x}_\bt(v)]=\overline{x}_{\al+\bt}(\eps
uv)$ for all $u,\,v\in\FF_q$ and some sign $\eps$ depending only on $\al$
and $\bt$ (see~\cite[Thm.~1.12.1(b)]{GLS}). Hence, $\overline{Q}$ is
isomorphic to $Q_{A_2}$, and the argument above gives
$X_{\al+\bt}\leq \ker(\chi)$. Since $[Q,Q]=X_{\al+\bt}X_{2\al+\bt}$, the
result follows. In the following, we denote by $Q_{B_2}$ the group $Q$
considered here.

Assume that $R$ is of type $G_2$, and
$R^+=\{\al,\,\bt,\,\al+\bt,\,2\al+\bt,\,3\al+\bt,\,3\al+2\bt\}$. Since
$p\neq 3$, we deduce from~\cite[Thm.~1.12.1]{GLS} that the subgroup
$Y=X_\al X_{2\al+\bt}X_{3\al+\bt}$ is isomorphic to $Q_{A_2}$ (with
$Q_{A_2}\rightarrow Y$ defined by
$x_\al(u)x_\bt(v)x_{\al+\bt}(w)\mapsto x_\al(u)x_{2\al+\bt}(v)
x_{3\al+\bt}(3w)$), and as above, we obtain $X_{3\al+\bt}\leq \ker(\chi)$.
Similarly, $X_{3\al+2\bt}\leq\ker(\chi)$. Note that $X_{3\al+\bt}X_{3\al+2\bt}$
is normal in $Q$, and $\overline{Q}=Q/X_{3\al+\bt}X_{3\al+2\bt}$ is isomorphic
to $Q_{B_2}$ by the commutator relations. Since $\chi$ factorizes through
$\overline{Q}$, and because $p\neq 2$, we deduce in a similar way that
$\ker(\chi)$ contains $[\overline{Q},\overline{Q}]=\overline{X}_{\al+\bt}
\overline{X}_{2\al+\bt}$. The result follows.

Now, we consider the case of twisted groups. Suppose that $R$ is of type
$A_2$. By~\cite[Thm.~2.4.5]{GLS}, $Q$ is isomorphic to the untwisted group
$Q_{A_2}$ over $\FF_{q^2}$. In particular, the minimal degree of a non-linear
character of $Q$ is $q^2$ and the result follows as above.

Suppose that $R$ is of type $B_2$. Then there are two possibilities. Either
$Q$ is a Sylow $p$-subgroup of $\tw2A_3(q)$ or of $\tw2A_4(q)$. We first
assume that $\bG$ is of type $A_3$. Let
$\Delta=\{\al_1,\,\al_2,\,\al_3\}$ be a system of simple roots of $\bG$.
Without loss of generality, we can assume that $\al=\pi(\al_1)$ and
$\bt=\pi(\al_2)$. We have $\Phi_{\theta(\al_1)}=\{\al_1,\al_3\}$,
$\Phi_{\theta(\al_2)}=\{\al_2\}$,
$\Phi_{\theta(\al_1+\al_2)}=\{\al_1+\al_2,\,\al_2+\al_3\}$ and
$\Phi_{\theta(\al_1+\al_2+\al_3)}=\{\al_1+\al_2+\al_3\}$,
where $\theta$ is the map defined in~(\ref{eq:thetadef}) above.
In particular, the root subgroups $X_\al$ and $X_{\al+\bt}$ are isomorphic to
$\FF_{q^2}$, and $X_\bt$ and $X_{2\al+\bt}$ to $\FF_q$.
Consider $L=Y_\al Y_{\al+\bt}X_{2\al+\bt}$, where $Y_\al=x_\al(\FF_q)$ and
$Y_{\al+\bt}=x_{\al+\bt}(\FF_q)$. By~\cite[Thm.~2.4.5(b)(2)]{GLS}, $L$ is a
subgroup of $Q$ and we have
$[x_\al(u),x_{\al+\bt}(v)]=x_{2\al+\bt}(2\eps uv)$ for all
$u,\,v\in\FF_{q}$. In particular, $L$ is isomorphic to $Q_{A_2}$ (because
$p\neq 2$) and we conclude as above that $X_{2\al+\bt}\leq \ker(\chi)$.
Note that $X_{2\al+\bt}$ is normal in $Q$, and by~\cite[Thm.~2.4.5(b)(3)]{GLS}
the quotient $\overline{Q}=Q/X_{2\al+\bt}$ is isomorphic to
$\overline{X}_\al\overline{X}_\bt\overline{X}_{\al+\bt}$ with the commutator
relation $[\overline{x}_\al(u),\overline{x}_\bt(v)]=
\overline{x}_{\al+\bt}(\eps u^qv)$. In particular, $\overline{Q}$ is
isomorphic to the group $Y$ of Lemma~\ref{lem:Aodd}. An isomorphism between
$Y$ and $\overline{Q}$ is given by $x_\al(u)x_\bt(v)x_{\al+\bt}(w)\mapsto
\overline{x}_\al(u)\overline{x}_\bt(v)\overline{x}_{\al+\bt}(\eps w^q)$.
So, the degrees of the irreducible characters of $\overline{Q}$ are $1$ or
$q$, and the same argument as above gives that $X_{\al+\bt}\leq\ker(\chi)$,
as required.

Consider now the case that $\bG$ is of type $A_4$, with positive roots
$$\Phi^+=\{\al_1,\al_2,\al_3,\al_4,\al_1+\al_2,\al_2+\al_3, \al_3+\al_4,
  \al_1+\al_2+\al_3,\al_2+\al_3+\al_4,\al_1+\al_2+\al_3+\al_4\},$$
and $Q$ is a Sylow $p$-subgroup of $\tw2A_4(q)$. Choose notation so that
$\bt=\pi(\al_1)$ and
$\al=\pi(\al_2)$. The orbits under the relation $\sim$ are $\{\al, 2\al\}$,
$\{\bt\}$, $\{\al+\bt,2(\al +\bt)\}$ and $\{2\al+\bt\}$. We choose
$\al$, $\bt$, $\al+\bt$ and $2\al+\bt$ as representatives.
Note that $\Phi_\al$ and $\Phi_{\al+\bt}$ are of type $A_2$, and $\Phi_\bt$
and $\Phi_{2\al+\bt}$ are of type $A_1\times A_1$. Thus, $X_\al$ and
$X_{\al+\bt}$ have $q^3$ elements and are isomorphic to a Sylow $p$-subgroup of
$\tw2A_2(q)$. We consider the groups $L=X_\al X_{\al+\bt}X_{2\al+\bt}$, and
$L'=X_{\al'}X_{\bt'}X_{\al'+\bt'}$, where the $X_{\gamma}$'s for
$\gamma\in\{\al',\bt',\al'+\bt'\}$, are the root subgroups associated
to a positive root system of type $A_2$ over $\FF_{q^2}$. Define
$$\varphi:L\rightarrow L',\
  x_\al(t,u)x_{\al+\bt}(t',v)x_{2\al+\bt}(w)\mapsto
  x_{\al'}(t)x_{\bt'}(t')x_{\al'+\bt'}(w),$$
for all $t,\,t',\,u,\,v\in\FF_{q^2}$ such that $t^{q+1}=\eps(u+u^q)$ and
$t'^{q+1}=\eps'(v+v^{q})$, where $\eps$ (resp.~$\eps'$) is a sign
depending only on $\al$ (resp.~$\al+\bt$). Then, by~\cite[Table 2.4 and
Thm.~2.4.5(b)(1), (c)(1)]{GLS}, $\varphi$ is a surjective group
homomorphism with kernel $[X_\al,X_\al][X_{\al+\bt},X_{\al+\bt}]$. Furthermore,
we derive
from~\cite[Table~2.1]{GecSU3} that the irreducible characters of $X_\al$ and
$X_{\al+\bt}$ have degree $1$ or $q$. Indeed, Geck \cite{GecSU3} computed
the character table of a Borel subgroup $B$ of $\tw2A_2(q)$, and by
Clifford theory, the degrees of the irreducible characters of the Sylow
$p$-subgroup contained in $B$ are the $p$-parts of the degrees of the
irreducible characters of $B$. So, by the same argument as above, $\ker(\chi)$
contains $[X_\al,X_\al][X_{\al+\bt},X_{\al+\bt}]$, $\chi$ factorizes through
$L'$, and we obtain that $X_{\al'+\bt'}$ lies in $\ker(\chi)$. This proves that
$X_{2\al+\bt}\leq \ker(\chi)$. Since $X_{2\al+\bt}\lhd Q$, we can work in
the quotient $Q/X_{2\al+\bt}$, and using the relations
\cite[Thm.~2.4.5(b)(1), (c)(2)]{GLS}, the same argument gives that
$X_{\al+\bt}\le\ker(\chi)$. The result follows.

Finally, assume that $R$ is of type $G_2$. Then $Q$ is a Sylow $p$-subgroup
of $\tw3D_4(q)$, and we derive the result from~\cite[Table A.6]{Him04}.
\end{proof}

\begin{rem}   \label{rem:condprime}
Note that the argument of the previous proof also applies for $G$ untwisted
and $R$ of type $B_2$ with $p\neq 2$, and for $G$ any twisted group such that
$R$ is not of type $G_2$ or $B_2$ (the case that $Q$ is a Sylow $p$-subgroup
of type $\tw2A_3(q)$).
\end{rem}

Recall the definition of $m(G,p)$ from Section~\ref{subsec:2.2}.

\begin{prop}   \label{prop:defgr}
 Let $G=\bG^F$ be a quasi-simple group of Lie type not of type $A_1$, and $P$
 a Sylow $p$-subgroup of $G$. Then $\mh(P)=m(G,p)$.
\end{prop}

\begin{proof}
When $G$ is a Suzuki or a Ree group, the result holds by~\cite[p.~126]{Suz},
\cite[Lemma 5]{Eat} and~\cite[Table~1.5]{HH09}. For $G=\tw2A_2(q)$, we have
already argued in Lemma~\ref{lem:rktwo} that the result holds
by~\cite[Table~2.1]{GecSU3}. For $G=G_2(q)$ with $q=2^f$ or $q=3^f$, we derive
the claim from \cite[Table~I-2]{EY86} and~\cite[Table~I-2]{EnoG2}, and for
$G=\tw3D_4(q)$, the result follows from \cite[Table A.6]{Him04} for $p$ odd,
and from~\cite[Table~A.6]{Him07} for $p=2$. So from now on, we assume that
$G$ is none of the groups considered above and in particular $\Sigma$ has
rank at least $2$.

By~\cite{Hw73}, we have $[P,P]=\prod_{\wal\mid\al\notin\Delta}X_\wal$ whenever
\begin{equation}   \label{eq:exc-derive}
  G\notin\{B_n(2),\,C_n(2),G_2(3), F_4(2), \tw2B_2(2), {}^2G_2(3), \tw2F_4(2)\}.
\end{equation}
First, we suppose that $p>3$. Let $\chi\in\Irr(P)$ be such that $\chi(1)<q$.
Let $\gamma'\in\Phi^+\backslash \Delta$. Write $\gamma=\pi(\gamma')$. By the
proof of~\cite[Cor.~B.2]{MT}, there are linearly independent positive
roots $\al,\,\bt\in\Sigma$ such that $\gamma=\al+\bt$ and hence
$\hgt(\gamma)> \max\{\hgt(\al),\hgt(\bt)\}$, where $\hgt$ denotes the
height function on the root system. In particular, $X_{\gamma}$ lies in the
derived subgroup of $\langle X_\al,X_\bt\rangle$. Now, applying
Lemma~\ref{lem:rktwo}, we conclude that $X_{\gamma}\leq\ker(\chi)$. It follows
that $[P,P]\leq \ker(\chi)$, and $\chi$ is a linear character of $P$. So,
if $\chi$ is a non-linear irreducible character of $P$, then $\chi(1)\geq q$.
By Remark~\ref{rem:condprime} this also covers the case $p=3$.

Assume $p=2$. Then by~\cite[Tables~A.1--A.4]{AHH}, the minimal degree of a
non-linear character of a Sylow $2$-subgroup of $B_2(q)$ is $q/2$.
Suppose now that $G=C_n(q)$ with $q>2$ even and $\chi\in\Irr(P)$ with
$\chi(1)<q/2$.
(Recall that $B_n(q)\simeq C_n(q)$ by~\cite[Thm.~2.2.10]{GLS}.) Note that
the previous argument does not apply directly. Indeed, whenever a long root is
the sum of two short roots $\al$ and $\bt$, then $X_{\al+\bt}$ commutes with
$X_\al$ and $X_\bt$; see~\cite[Thm.~1.12.1 (b)(1)]{GLS}. However,
\cite[Rem.~1.8]{GLS} implies that in a root system of type $C_n$ with
basis $\Delta$, every positive long root $\gamma\notin \Delta$ belongs to a
root system $\Psi$ of type $B_2$ with basis $\al,\,\bt$ such that $\al$ is a
short root and $\bt$ is the long root of $\Delta$. Restricting $\chi$ to the
subgroup $\prod_{\delta\in\Psi^+}X_{\delta}$, we deduce from the case
$B_2(q)$ that $X_{\gamma}$ lies in $\ker(\chi)$. Furthermore, for $\gamma$
a positive short root such that $\gamma\notin\Delta$, the argument for $p>3$
applies. It follows that $[P,P]\leq\ker(\chi)$, and $\chi$ is linear.

Suppose next that $G=\tw2A_{2n+1}(q)$ and $p=2$. Then the root system $\Sigma$
is also of type $C_{n+1}$. We conclude in the same way that if $\chi$ is a
non-linear character of $P$, then $\chi(1)\geq q$, remarking that the minimal
degree of a non-linear character of a Sylow $2$-subgroup of $\tw2A_3(q)$
is $q$ by~\cite[Tables~A.6,\,A.8,\,A.9 and~A.10]{AHH}.

Assume $G=F_4(q)$ with $q>2$ even, and let $\chi$ be an irreducible
character of $P$ such that $\chi(1)<q/2$. Let $\Phi$ be the root system of
$G$ as in~\cite[Rem.~1.8]{GLS}. Denote by
$$\Phi^+=\Big\{\veps_i\,(1\leq i\leq 4),\,\veps_i\pm\veps_j\,(1\leq i<j\leq 4),
  \, \frac{1}{2}(\veps_1\pm\veps_2\pm\veps_3\pm\veps_4) \Big\}$$
the positive roots with respect to the basis
$$\Delta=\{\veps_2-\veps_3,\veps_3-\veps_4,\veps_4,
  \frac{1}{2}(\veps_1-\veps_2-\veps_3-\veps_4)\}.$$
Let $i<j$. Then the positive long root $\veps_i+\veps_j$ lies in the subsystem
$\Psi$ of type $B_2$ with basis $\{\veps_i-\veps_j,\veps_j\}$.
In particular, restricting $\chi$ to $\prod_{\delta\in\Psi^+}X_{\delta}$, we
conclude as in the case of type $C_n$ in characteristic $2$ that
$X_{\veps_i+\veps_j}$ lies in $\ker(\chi)$. Now, suppose $2\leq j\leq 4$, and
let $2\leq k<l\leq 4$ be such that $\{j,k,l\}=\{2,3,4\}$. Then
$\al=\veps_k+\veps_l$ and $\bt=\frac{1}{2}(\veps_1-\veps_j-\veps_k-\veps_l)$
are in $\Phi^+$, and $\veps_1-\veps_j$ lies in the root subsystem of type
$B_2$ with basis $\{\al,\bt\}$. As previously, we conclude that
$X_{\veps_1-\veps_j}$ lies in $\ker(\chi)$. For the remaining non-simple
positive roots, the argument for $p>3$ can be applied (note that the positive
long root $\veps_2-\veps_4=(\veps_2-\veps_3)+(\veps_3-\veps_4)$ is the sum of
two positive long roots). This proves that $\chi$ is linear.

It remains to prove that $P$ has an irreducible character with the announced
degree.
First, recall from the proof of Lemma~\ref{lem:rktwo} that if $G=A_2(q)$, then
$P$ has an irreducible character of degree $q$. Let $G=B_2(q)$. If $p\neq 2$,
then $P/X_{2\al+\bt}$ is isomorphic to a Sylow $p$-subgroup of $A_2(q)$.
Inflating to $P$ a character of degree $q$ of $P/X_{2\al+\bt}$, we conclude
that $P$ has an irreducible character of degree $q$. If $p=2$, then $P$ has
an irreducible character of degree $q/2$ by~\cite{EY86}. Similarly, if
$G=G_2(q)$ and $p>3$, then $P/X_{2\al+\bt}X_{3\al+\bt}X_{3\al+2\bt}$ is
isomorphic to a Sylow $p$-subgroup of $A_2(q)$. Thus, $P$ has an irreducible
character of degree $q$.
Let $G=\tw2A_3(q)$. If $p\neq 2$, then we have shown in the proof of
Lemma~\ref{lem:rktwo} that $P$ has a quotient isomorphic to the group $Y$
of Lemma~\ref{lem:Aodd}. In particular, $P$ has irreducible characters of
degree $q$. If $p=2$, then we conclude
with~\cite[Table A.6,\,A.8,\,A.9 and A.10]{AHH}.
Let $G=\tw2A_4(q)$, and $\al$ be the simple short root of $\Sigma^+$.
Then $X=\prod_{\gamma\in\Sigma^+\setminus\{\al\}}X_{\gamma}\lhd P$, and $P/X$
is isomorphic to $X_\al$ which has an irreducible character
of degree $q$ by~\cite[Table~2.1]{GecSU3} (because $X_\al$ is isomorphic to
a Sylow $p$-subgroup of $\tw2A_2(q)$).

Suppose now that $G$ is any quasi-simple group of Lie type as in
Lemma~\ref{lem:rktwo} or in Remark~\ref{rem:condprime}, or $p=2$ and
$G=C_n(q)$ or $F_4(q)$ with $q>2$. Let $\al$ and $\bt$ be two simple roots
generating a subgroup of type $A_2$ (when $G$ is untwisted of type
$A_n,\,D_n,\,E_6,\,E_7,\,E_8$) or $B_2$ (when $G$ is untwisted of type $B_n$,
$C_n$, $F_4$, or $G=\tw2A_n(q)$, $\tw2D_n(q)$ or $\tw2E_6(q)$).
Let $\Psi^+$ be the positive roots generated by $\al$ and $\bt$. Then
$X=\prod_{\gamma\in\Phi^+\setminus\Psi^+}X_{\gamma}\lhd P$ and $P/X$
is isomorphic to $\langle X_\al,X_\bt\rangle$. By the above discussion,
$P$ has an irreducible character of the required degree.

Finally, we consider the exceptions in (\ref{eq:exc-derive}). The only
cases to treat are $G=C_n(2)$ and $G=F_4(2)$. But then $P$ has a quotient
isomorphic to a Sylow $2$-subgroup of $B_2(2)$, which has an irreducible
character of degree $q=2$. This completes the proof.
\end{proof}

\begin{rem}
 Kazhdan \cite[Prop.~2]{Ka77} has shown that for large enough primes $p$, all
 character degrees of Sylow $p$-subgroups of finite reductive groups in
 characteristic~$p$ are polynomials in $q$. This gives another approach for
 the lower bound on $\mh(P)$ at least when $p$ is large.
\end{rem}

\subsection{The Eaton--Moret\'o conjecture for quasi-simple groups of Lie
type in defining characteristic}

\begin{thm}   \label{thm:Liep}
 Let $S$ be a quasi-simple group of Lie type in characteristic $p$ and $B$
 a $p$-block of $S$ of positive defect. Then the Eaton--Moret\'o conjecture
 holds for $B$.
\end{thm}

\begin{proof}
The character table of $\tw2F_4(2)'$ is known and one finds that it has a
single $2$-block of positive defect, which contains a character of degree~26,
and its Sylow 2-subgroup has an irreducible character of degree~2. For all
other cases, let us first assume that $|Z(S)|$ is prime to $p$. In particular,
$S$ is not an exceptional covering group of its non-abelian composition factor.
Thus, $S$ is a quotient of some finite group $G$ of simply connected Lie type.
As any irreducible character of $S$ can be considered as a character of $G$,
we may assume that $S=G$ is of simply connected type. Thus we are in the
situation of Proposition~\ref{prop:Liep}. In particular $B$ is of full defect,
so a Sylow $p$-subgroup of $G$ is a defect group for $B$. Now compare
Proposition~\ref{prop:Liep} with Proposition~\ref{prop:defgr}. \par
Finally, assume that $G$ is an exceptional covering group of one of
$$\begin{aligned}
\PSL_2(4),\PSL_2(9),\PSL_3(2),\PSL_3(4),\PSL_4(2),\PSU_4(2),\PSU_4(3),
  \PSU_6(2),\\
  \Sp_6(2),\OO_7(3),\OO_8^+(2),\tw2B_2(8),G_2(3),G_2(4),F_4(2),\tw2E_6(2)
\end{aligned}$$
(see \cite[Tab.~24.3]{MT}). Note that the exceptional part of the Schur
multiplier always has order a power of the defining prime $p$. Thus, the block
$B$ lifts a block $B'$ of $G/O_p(G)$ of the same defect, and for that
we already showed the conjecture before. In all cases but
$$\PSL_2(4),\PSL_2(9),\PSL_3(4),$$
the minimal non-zero height in $B'$ was equal to~1, so cannot become smaller
in $B$. The alternating groups $\PSL_2(4)\cong\fA_5$ and $\PSL_2(9)\cong\fA_6$
were already treated in Theorem~\ref{thm:alt}.
The exceptional covering groups $2.\PSL_3(4)$ and $6.\PSL_3(4)$ have faithful
irreducible characters of degrees~10, respectively~6, and direct calculation
shows that their Sylow 2-subgroups have twelve irreducible characters of
degree~2. This completes the proof.
\end{proof}

\section{Groups of Lie-type: cross characteristic} \label{sec:Liecross}

We now discuss some cases in groups of Lie type in cross characteristic.
So here $G$ is defined in characteristic~$r$, and we consider $p$-blocks of $G$
for primes $p\ne r$.

\subsection{Unipotent blocks in groups of exceptional type}

We first deal with exceptional groups of Lie type.

\begin{prop}   \label{prop:smallexc}
 Let $p$ be a prime and $G$ be quasi-simple with $G/Z(G)$ one of
 $$\tw2B_2(q^2),{}^2G_2(q^2),G_2(q),\tw3D_4(q),\tw2F_4(q^2).$$
 Then the Eaton--Moret\'o conjecture holds for all $p$-blocks of $G$.
\end{prop}

\begin{proof}
Since Brauer's height zero conjecture has been proved for blocks of
quasi-simple groups with abelian defect groups (see \cite{KM13}) we may assume
that $G$ has non-abelian Sylow
$p$-subgroups. The case where $p$ is the defining prime was already handled
in Theorem~\ref{thm:Liep}, thus now assume that we are in cross characteristic.
For groups $\tw2B_2(q^2)$ and ${}^2G_2(q^2)$ all such Sylow subgroups are
abelian. For $G_2(q)$ we only need to consider $p=2,3$. For $p=3$ only the
principal block has non-abelian defect, and it contains characters of
height~1 (see \cite[\S2.2,2.3]{HS90}). Since the Sylow 3-subgroups are
non-abelian extensions of a homocyclic abelian group with a group of order~3,
they also possess characters of height~1. For $p=2$, the blocks with
non-abelian defect groups are described in \cite{HS92}: these are the blocks
denoted $B_1,B_3,B_{1a},B_{1b},B_{2a}$ and $B_{2b}$. All of them contain
characters of height~1. The defect groups are either the Sylow 2-subgroups
of $G_2(q)$, or semidihedral, thus also possess characters of height~1.
\par
For $\tw3D_4(q)$ the relevant primes are again $p=2,3$. By
\cite[Prop.~5.4]{DM87} only the principal 3-block has non-abelian defect
groups, and the characters $\chi_{5,1}$ (if $q\equiv1\pmod3$) respectively
$\chi_{10,1}$ have height~1. The Sylow 3-subgroups possess an abelian normal
subgroup of index~3, so have characters of height~1 as well. For the prime
$p=2$ there exist three types of non-abelian defect groups, by
\cite[Prop.~5.3]{DM87}, and all these have characters of height~1. The
principal block will be discussed in Proposition~\ref{prop:B0p=2}.
For the other 2-blocks Jordan decomposition gives a height preserving
bijection to a unipotent block of a group ${}^{(2)}\!A_2(q)$, $A_1(q)$ or
$A_1(q^3)$, where the existence of characters of height~1 is easily checked.
For $\tw2F_4(q^2)$ we just need to consider $p=3$. By \cite[Bem.~2]{Ma90}
only the principal 3-block has non-abelian defect groups, and it contains
characters of height~1. The Sylow 3-subgroup of $G$ is contained inside
a subgroup $\SU_3(q^2)$ and thus easily seen to possess characters of height~1
as well.
\end{proof}

To deal with the groups of large rank, we need some information on heights
in Sylow $p$-subgroups.

\begin{prop}   \label{prop:principalblock}
 Let $G$ be a quasi-simple group of Lie type in characteristic $r$, and
 $p\neq r$ be a prime number. Let $P$ be a Sylow $p$-subgroup of $G$. If
 $P$ is non-abelian, then $P$ has an irreducible character of degree $p$.
\end{prop}

\begin{proof}
For $p=2$, this will be proved in Proposition~\ref{prop:B0p=2}. Assume
$p\geq 3$, and let $P$ be a non-abelian Sylow $p$-subgroup of $G$. First
consider the case that $G$ is of classical type. Then, by~\cite{Wei},
there exist integers $a\geq 1$ and $m\geq 1$ (because
$P$ is non-abelian) such that the group
$$\overline{P}\simeq
\ZZ/p^a\ZZ\wr\underbrace{\ZZ/p\ZZ\wr\cdots\wr\ZZ/p\ZZ}_{m\
\textrm{times}}$$
is a quotient of $P$. If $m=1$, then $\overline{P}$ has an irreducible
character of degree $p$ by Clifford theory. If $m\geq 2$, then
$\ZZ/p\ZZ\wr\ZZ/p\ZZ$ is a quotient of $\overline P$ that possesses an
irreducible character of degree $p$ (again by Clifford theory). This proves
that $P$ has an irreducible character of degree $p$, as required.

Now, suppose that $G$ is of exceptional $E$- or $F$-type. Here we only need
to consider primes $p\le7$ (respectively $p\le 5$ when $G=\,{}^{(2)}\!E_6(q)$,
respectively $p=3$ when $G=F_4(q)$) because any other Sylow $p$-subgroups of
$G$ are abelian by~\cite[Thm. 25.16]{MT}. When $p=5$ for $E_6(q)$, $\tw2E_6(q)$
or $E_7(q)$, or when $p=7$ for $E_7(q)$ or $E_8(q)$, then a Sylow $p$-subgroup
$P$ of $G$ contains an abelian normal subgroup of index~$p$, hence $\mh(P)=1$.
In the remaining case that $p=3$, or that $G=E_8(q)$ with $p=5$, we give in
Table~\ref{tab:syl} a subsystem subgroup $H$ of $G$ that contains a Sylow
$p$-subgroup of $G$ in which we may detect an irreducible character of
height~1. Indeed,
consider for example $G=F_4(q)$. Here, for $q\equiv1\pmod3$ a Sylow
3-subgroup of $G$ is contained in the centralizer of a 3-element, of type
$A_2^2$. Taking the quotient of this subgroup by a normal subgroup $\SL_3(q)$
we obtain a group $\PGL_3(q)$, whose Sylow 3-subgroup is non-abelian, hence
has an irreducible character of height~1. So the same holds for $P$. When
$q\equiv2\pmod3$, the analog argument goes through with the subgroup
$\tw2A_2\tw2A_2$ with a quotient of type $\PGU_3(q)$. A similar argument
applies for the other types.

\begin{table}[htbp]
\caption{Overgroups of Sylow subgroups}   \label{tab:syl}
\[\begin{array}{|r|r|ll|ll|}
\hline
 G& p& H& & H& \\
\hline
    F_4(q)& 3& A_2(q)^2& (q\equiv1\pmod3)& \tw2A_2(q)^2& (q\equiv2\pmod3)\\
    E_6(q)& 3& A_2(q)^3& (q\equiv1\pmod3)& A_2(q^2)\,\tw2A_2(q)& (q\equiv2\pmod3)\\
\tw2E_6(q)& 3& A_2(q^2)A_2(q)& (q\equiv1\pmod3)& \tw2A_2(q)^3& (q\equiv2\pmod3)\\
    E_7(q)& 3& E_6(q)A_1(q)& (q\equiv1\pmod3)& \tw2E_6(q)A_1(q)& (q\equiv2\pmod3)\\
    E_8(q)& 3& E_6(q)A_2(q)& (q\equiv1\pmod3)& \tw2E_6(q)\,\tw2A_2(q)& (q\equiv2\pmod3)\\
          & 5& A_4(q)^2& (q\equiv1\pmod5)&  \tw2A_4(q)^2& (q\equiv4\pmod5)\\
          & 5&       \tw2A_4(q^2)& (q\equiv2,3\pmod5)& & \\
\hline
\end{array}\]
\end{table}
\end{proof}

\begin{prop}   \label{prop:B0exc}
 The Eaton--Moret\'o-conjecture holds for the principal blocks of exceptional
 type quasi-simple groups.
\end{prop}

\begin{proof}
Let $G$ be quasi-simple of exceptional Lie type, and $p$ a prime. We may
assume that $p$ is different from the defining characteristic of $G$ by
Theorem~\ref{thm:Liep}, and also that Sylow $p$-subgroups of $G$ are
non-abelian by \cite{KM13}. Thus $p$ divides the order of the Weyl group
of $G$. We may
further assume by Proposition~\ref{prop:smallexc} that $G$ is of Lie rank at
least~4. The case $p=2$  will be settled in Proposition~\ref{prop:B0p=2} by
a general argument. So assume that $p\ge 3$. We claim that $\mh(B_0)=1=\mh(P)$
whenever $G$ is simple, where $P$ is a Sylow $p$-subgroup of $G$. Then the
claim automatically follows for all covering groups of $G$. The fact that
$\mh(P)=1$ is shown in
Proposition~\ref{prop:principalblock}. We go through the individual cases.
The unipotent characters in the principal block are described in
\cite[Thm.~A(c)]{En00} (see also the tables on p.349--358 in loc.~cit.).
For $G=F_4(q)$ only $p=3$ matters. Here, the principal block contains the
unipotent character $\phi_{12,4}$ if $q\equiv1\pmod3$, respectively
$F_4^{II}[1]$ if $q\equiv2\pmod3$, of 3-defect~1. For $G=E_6(q)$ and $p=3$,
the unipotent character $\phi_{6,1}$ has height~1 and lies in the principal
block when $q\equiv1\pmod3$, while for $q\equiv2\pmod3$, the character
$\phi_{20,10}$ is as required. For $\tw2E_6(q)$, we may take the Ennola duals
of the above characters, viz.~$\phi_{2,4}'$ when $q\equiv2\pmod3$
respectively $\phi_{12,4}$ when $q\equiv1\pmod3$. \par
For $E_7(q)$ we need to look at $p=3$, and at $p=5,7$ when $q\equiv\pm1\pmod p$.
For $p=3$, the unipotent character $\phi_{21,3}$ is of 3-height~1 in the
principal block. For $p=5$, the unipotent character $\phi_{210,6}$ has
5-height~1, for $p=7$ we may instead take the character $\phi_{7,1}$.
Finally, for $G=E_8(q)$ the relevant primes are $p=3,5$, and $p=7$ when
$q\equiv\pm1\pmod7$. For $p=3$, the character $\phi_{525,12}$ has 3-height~1,
while for $p=5,7$ we may choose $\phi_{35,2}$. This completes the proof.
\end{proof}

\begin{thm}   \label{thm:exc}
 The Eaton--Moret\'o-conjecture holds for all unipotent blocks of quasi-simple
 groups of exceptional Lie type.
\end{thm}

\begin{proof}
Let $G$ be a quasi-simple exceptional group of Lie type. By
Proposition~\ref{prop:smallexc} we may assume that $G$ has Lie rank at least~4.
Moreover, we can assume that $p$ is not the defining characteristic of $G$,
by Theorem~\ref{thm:Liep}. In addition we may assume that $p$ divides the
order of the Weyl group of $G$, since otherwise the Sylow $p$-subgroups of $G$
are abelian (see \cite[Thm.~25.14]{MT}). The unipotent blocks of $G$ and their
defect groups are described in \cite{En00}. By Proposition~\ref{prop:B0exc}
we need not consider principal blocks. The only non-principal unipotent
$p$-blocks with non-abelian defect groups are then those listed in
Table~\ref{tab:unpblocks}, lying above a $d$-cuspidal character of a $d$-split
Levi subgroup $L$ as indicated in column~4 of the table.

\begin{table}[htbp]
\caption{Non-principal unipotent blocks}   \label{tab:unpblocks}
\[\begin{array}{|r|r|r|c|l|l|}
\hline
 G& p& d& L& \text{conditions}& \chi\\
\hline
     E_6(q)& 3& 1& \Ph1^2.D_4(q)& q\equiv1\pmod3& *\\
\hline
 \tw2E_6(q)& 3& 2& \Ph2^2.D_4(q)& q\equiv2\pmod3& *\\
\hline
     E_7(q)& 2& 1&     \Ph1.E_6(q)& q\equiv1\pmod4& *\\
           & 2& 2& \Ph2.\tw2E_6(q)& q\equiv3\pmod4& *\\
           & 3& 1&   \Ph1^3.D_4(q)& q\equiv1\pmod3& D_4,r\\
           & 3& 2&   \Ph2^3.D_4(q)& q\equiv2\pmod3& \phi_{280,8}\\
\hline
     E_8(q)& 2& 1&     \Ph1^2.E_6(q)& q\equiv1\pmod4& E_6[\theta],\phi_{2,1}\\
           & 2& 2& \Ph2^2.\tw2E_6(q)& q\equiv3\pmod4& E_8[-\theta]\\
           & 3& 1&     \Ph1^4.D_4(q)& q\equiv1\pmod3& D_4,\phi_{12,4}\\
           & 3& 2&     \Ph2^4.D_4(q)& q\equiv2\pmod3& \phi_{448,25}\\
\hline
\end{array}\]
\end{table}

In all cases it is immediate from the description of defect groups $D$ in
\cite{En00} that $\mh(D)=1$. Except for the cases marked with a '*' we have
printed in the last column of Table~\ref{tab:unpblocks} a unipotent character
(in the notation of \cite[\S13.8]{Ca})
in the relevant block of height~1. In the first four cases, no such unipotent
character exists. But it was argued in the proof of \cite[Prop.~8.4]{KM13}
that there exists a character of height~1 in these blocks if and only if the
defect groups are non-abelian. This achieves the proof.
\end{proof}

\subsection{Principal blocks}

We now consider principal blocks of arbitrary quasi-simple groups and prove
Theorem~\ref{thm:main}(1).

\begin{prop}   \label{prop:B0p=2}
 Let $G$ be a finite quasi-simple group. Then the principal 2-block of $G$
 satisfies the Eaton--Moret\'o conjecture.
\end{prop}

\begin{proof}
We go through the various possibilities for $G$ according to the
classification. For $G/Z(G)$ sporadic, the claim was shown in
\cite[Thm.~D]{EM14}. For alternating groups and their covering groups, see
Theorem~\ref{thm:alt}. If $G$ is of Lie type in characteristic~2, then
we showed the claim in Theorem~\ref{thm:Liep}. So now assume that $G=\bG^F$
is of Lie type in odd characteristic. The Ree groups ${}^2G_2(q^2)$ have
abelian Sylow 2-subgroups, so we may assume that $G$ is not very twisted.
If $G$ is of type $A_1$, so either $G=\SL_2(q)$ or $\PSL_2(q)$ with $q$ odd,
then the characters in the principal 2-block of $G$ are those lying in Lusztig
series indexed by
2-elements $s$ in the dual group $G^*=\PGL_2(q)$ or $\PSL_2(q)$, and they are
of height~1 if and only if the centralizer $C_{G^*}(s)$ has index congruent to
$2\pmod4$ in $G^*$. Now the Sylow 2-subgroups of $G^*$ are dihedral,
thus such elements $s$ exist if and only if these Sylow 2-subgroups are
non-abelian, which is the case if and only if the Sylow 2-subgroups of $G$
are non-abelian. These in turn are quaternion or dihedral, hence possess
characters of height~1 if and only if they are non-abelian. This deals with
the case of groups of type $A_1$ (the exceptional covering groups of
$\PSL_2(9)\cong\fA_6$ were considered in Theorem~\ref{thm:alt}).
\par
In the remaining cases, let $B_0$ denote the principal 2-block of $G$. Let $q$
be defined as above, and set $d=1$ if $q\equiv1\pmod4$ and $d=2$ if
$q\equiv3\pmod4$. Let $\bT$ be an $F$-stable maximal torus of $\bG$ containing
a Sylow $\Phi_d$-subgroup.
Then the pair $(\bT,1)$ is $d$-cuspidal, and by~\cite[Thm.~A]{En00} the
$d$-Harish-Chandra series $\cE(G,(T,1))$ above the trivial character of
$T=\bT^F$ lies in $B_0$.  Moreover, by~\cite[Prop.~7.4]{MaH0}, there is a
bijection $\Irr(W)\rightarrow\cE(G,(T,1))$, $\phi\mapsto\chi_\phi$, with the
property that $\chi_\phi(1)\equiv\pm\phi(1)\pmod{\Phi_d(q)}$, where
$W=N_G(\bT)/T$ is the relative Weyl group of $T$. Since
$\Phi_d(q)\equiv0\pmod4$, this shows that $\phi\in\Irr(W)$ has $2$-height~$1$
if and only if $\chi_\phi$ has.
\par
On the other hand by~\cite[4.10.2]{GLS}, $N_G(T)$ contains a Sylow $p$-subgroup
$P$ of $G$. Then $P_T=T\cap P$ is a normal subgroup of $P$ such that
$W_2=P/P_T$ is isomorphic to a Sylow $2$-subgroup of $W$. Thus, we are done
if we can show that both $W$ and $W_2$ possess characters of height~$1$ if
$P$ is non-abelian. However, it is shown in~\cite[4.10.2]{GLS} that whenever
$p=2$, the torus $\bT$ is of type $1$ or $w_0$, where $w_0$ is the longest
element of the Weyl group of $\bG$. In particular,~\cite[Prop.~25.3]{MT} gives
that $W$ is isomorphic to a Weyl group.

If $G$ is of exceptional type, the claim can now be checked from the known
character tables. Suppose that $G$ is of classical type. Then $W$ is a
classical Weyl group, and has a quotient $W'$ isomorphic to the symmetric
group $\fS_m$ with $m\ge n/2$, where $n$ is the rank of $G$. Denote by
$\pi:W\rightarrow W'$ the canonical projection, and write $W'_2=\pi(W_2)$.
Then $W'_2$ is a Sylow $2$-subgroup of $W'$, and by~\cite[Thm.~4.3]{EM14},
whenever $m>2$ then $W'$ and $W'_2$ both have a character of $2$-height~$1$,
and the result follows. This still remains true when $m=2$ and $W_2$ is the
dihedral group of order~$8$ or the symmetric group $\fS_3$. \par
Finally, if $m=1$ and hence $S=\PSL_3(q)$ with $d=2$ or $S=\PSU_3(q)$ with
$d=1$, then the Sylow 2-subgroups contain an abelian normal subgroup of
index~2 and hence possess characters of height~1. On the other hand, the
irreducible Deligne--Lusztig characters of $S$ of degree~$q^3-1$, respectively
$q^3+1$, lying in the Lusztig series of a regular element of order~8 in a
torus of order $q^2-1$ also have height~1 and are contained in the principal
block.
\end{proof}

The following result will be crucial for dealing with primes $p>2$:

\begin{lem}   \label{lem:ell-power}
 Let $p>2$ be a prime and $d$ an integer dividing $p-1$. Then for a primitive
 $d$th root of unity $\zeta$ we have
 $$\Phi_{dp^i}(\zeta)_p=p\quad\text{ for all $i\ge1$},$$
 and $\Phi_m(\zeta)$ is prime to $p$ for all $m\ne dp^i$.
\end{lem}

\begin{proof}
The fact that $\Phi_m(\zeta)$ is prime to $p$ for $m\ne dp^i$ is shown
for example in \cite[Lemma~5.2]{MaH0}. For the other assertion note that
$\Phi_{dp^i}(X)$ divides $\Phi_{p^i}(X^d)$ and that
$\Phi_{p^i}(X)=\Phi_p(X^{p^{i-1}})$ for $i\ge1$. Thus
$\Phi_{dp^i}(\zeta)$ divides $\Phi_{p^i}(1)=\Phi_p(1)=p$.
\end{proof}

\begin{thm}   \label{thm:B0}
 The Eaton--Moret\'o conjecture holds for the principal blocks of all
 quasi-simple groups.
\end{thm}

\begin{proof}
We go through the various possibilities for $G$ according to the
classification. For $G/Z(G)$ sporadic, the claim was shown in
\cite[Thm.~D]{EM14}. For alternating groups and their covering groups, see
Theorem~\ref{thm:alt} above. If $G$ is of Lie type in characteristic~$p$, then
we showed the claim in Theorem~\ref{thm:Liep}. If $G$ is of Lie type in odd
characteristic and $p=2$, the assertion is in Proposition~\ref{prop:B0p=2}.
Finally, if $G$ is of exceptional type, we proved this in
Proposition~\ref{prop:B0exc}.
\par
Thus we may assume that $G$ is of classical type, that $p\ne2$ is not the
defining characteristic of $G$ and that the Sylow $p$-subgroups of $G$ are
non-abelian. Let $G=\bG^F$ where $\bG$ is simple of simply
connected classical type. Let $d$ denote the order of $q$ modulo~$p$.

As we saw in Proposition~\ref{prop:principalblock}, any non-abelian Sylow
$p$-subgroup of $G$ has a character of $p$-defect $1$. We now will show
that the same property holds for the principal $p$-block of $G$. We consider
two cases.

Assume first that $p\geq 5$.
Then the principal $p$-block contains the principal $\Phi_d$-series $\cE(G,d)$
of unipotent characters of $G$ (see \cite[Rem.~6.12]{KM13} for references).
These are distinguished by the property that their generic degree polynomial
is not divisible by $\Phi_d$. Let $W$ denote the relative Weyl group of the
principal $\Phi_d$-series, that is, $W=N_G(\bL)/L$, where $\bL$ is the
centralizer of a Sylow $\Phi_d$-torus of $\bG$. Then by \cite[Thm.~3.2]{BrMaMi},
there is a bijection $\Irr(W)\rightarrow\cE(G,d)$, $\phi\mapsto\chi_\phi$, with
the property that $\phi(1)=\pm\chi_\phi(\zeta_d)$,
where $\chi_\phi(X)$ is the degree polynomial of the unipotent character
$\chi_\phi$. Now none of these degree polynomials is divisible by $\Phi_d$,
and then Lemma~\ref{lem:ell-power} in conjunction with \cite[Lemma~5.2]{MaH0}
shows that both $\chi_\phi(1)$ and $\phi(1)$ are divisible by the same power
of $p$.   \par
Furthermore, by \cite[\S3]{BrMaMi}, $W$ is either isomorphic to $G(d',1,m)$
for some positive integers $m$ and $d'$ (with $d'$ prime to $p$), or
to its subgroup $G(d',2,m)$ of index $2$. In all cases, $W$ has a quotient
isomorphic to $\fS_m$. On the other hand,
by~\cite[Thm.~25.14]{MT} we have that $p$ divides $|W|$ and since $d'$ is
prime to $p$, it follows that $p$ divides $m!$. Let $P$ be a Sylow
$p$-subgroup of $\fS_m$. If $P$ is non-abelian, then by~\cite[Note
p.74]{Ol93}, the principal $p$-block of $\fS_m$ contains an irreducible
character $\phi$ of height $1$, so $\phi$ has $p$-defect $1$.  Suppose
now that $P$ is abelian, and write $m=a_0+a_1p$ for the $p$-adic
decomposition of $m$ (where $0\leq a_0,\,a_1<p$, and $a_1\neq 0$ because
$p\leq m$). Set $\alpha_0=a_0+p$ and $\alpha_1=a_1-1\geq 0$. Then we have
$\alpha_0+\alpha_1p=m$ and $\alpha_0+\alpha_1-a_0-a_1=p-1$. Hence, following
\cite[p.\,43]{Ol93}, the $p$-extension $(\alpha_0,\alpha_1)$ of $m$ has
deviation~$1$. With the notation (5) of \cite[p.~43]{Ol93}, we have
$c_{p}(p,\alpha_1)\neq 0$ because $p>\alpha_1$. Furthermore, one has
$c_{p}(1,\alpha_0)=c(p,\alpha_0)$, where $c(p,\alpha_0)$ denotes the number
of $p$-cores of $\alpha_0$.
Assume $a_0\neq 0$. Then the partition $(a_0+1,1^{p -1})$ has no $p$-hooks.
It follows that $c(p,\alpha_0)\neq 0$. (Note that until now, we did not
need the assumption $p\neq 3$). Assume $a_0=0$. Then any non-hook partition
is a $p$-core (because $p\geq 5$), and we also have $c(p,\alpha_0)\neq 0$.
In these cases, \cite[Prop.~6.6]{Ol93} gives that $\fS_m$ has an irreducible
character of $p$-defect $1$, and so, the principal $p$-block of~$G$.
\par
Suppose now that $p=3$ and $a_0=0$, i.e., $m\in\{3,6\}$ and $d\in\{1,2\}$.
If $d'=2$ then $W\cong G(2,1,m)$ or $G(2,2,m)$ has a character of
$3$-height~$1$, and we conclude by the same argument as above. Assume $d'=1$.
Thus, $W\cong \fS_m$ with $m\in\{3,6\}$. On the other hand, \cite[\S 3]{BrMaMi}
gives that $G$ is of type $A$, isomorphic to $\SL_3(q)$ or $\SL_6(q)$ (resp.
$\SU_3(q)$ or $\SU_6(q)$) whenever $d=1$ (resp. $d=2$). Suppose that $G$
is untwisted, that is, $3$ divides $q-1$ and $G=\SL_3(q)$ or $G=\SL_6(q)$
(according to $m=3$ or $m=6$).  Let $\bT$ be an $F$-stable maximally split
torus of $\bG$, and $s$ be a semisimple element of order $3$ of $\bT^{*F}$
such that $C_{\bG^*}(s)=\bT^*$. Then the unique irreducible character of
$\cE(G,s)$ has $3$-defect~$1$ and lies in the principal $3$-block of $G$
by~\cite[Main Thm.]{CaEn}. When $G$ is twisted, that is $G\cong\SU_3(q)$ or
$\SU_6(q)$ with $q+1$ divisible by $3$, we conclude by a similar argument.
\par
We need not consider exceptional covering groups of $G/Z(G)$, since the
order of their centers differ from that of some non-exceptional covering group
of $G$ by a power of the defining characteristic.
\end{proof}



\begin{thebibliography}{131}

\bibitem{AHH}
{\sc J. An, F. Himstedt, S. Huang}, Dade's invariant conjecture for the
  symplectic group $\Sp(4,2^n)$ and the special unitary group $\SU_4(2^{2n})$
  in defining characteristic. Comm. Algebra {\bf 38} (2010), 2364--2403.

\bibitem{BrMaMi}
{\sc M. Brou\'e, G. Malle, J. Michel}, Generic blocks of finite reductive
  groups. Repr\'esentations unipotentes g\'en\'eriques et blocs des groupes
  r\'eductifs finis. Ast\'erisque No. 212 (1993), 7--92.

\bibitem{CaEn}
{\sc M. Cabanes, M. Enguehard}, On unipotent blocks and their ordinary
  characters. Invent. Math. {\bf117} (1994), 149--164.

\bibitem{Ca}
{\sc R.W. Carter}, \emph{Finite Groups of Lie Type}. Wiley-Interscience,
  New York, 1985.

\bibitem{DM87}
{\sc D.I. Deriziotis, G.O. Michler}, Character table and blocks of finite
  simple triality groups $\tw3D_4(q)$. Trans. Amer. Math. Soc. {\bf303} (1987),
  39--70.

\bibitem{Eat}
{\sc C. Eaton}, Dade's inductive conjecture for the Ree groups of type
  $G_2$ in the defining characteristic. J. Algebra {\bf226} (2000), 614--620.

\bibitem{EM14}
{\sc C. Eaton, A. Moret\'o}, Extending Brauer's height zero conjecture to
  blocks with nonabelian defect groups. Int. Math. Res. Not., to appear,
   doi: 10.1093/imrn/rnt131.

\bibitem{En00}
{\sc M. Enguehard}, Sur les $l$-blocs unipotents des groupes r\'eductifs
  finis quand $l$ est mauvais. J. Algebra {\bf230} (2000), 334--377.

\bibitem{EnoG2}
{\sc H. Enomoto}, The characters of the finite Chevalley group $G_2(q)$,
  $q=3^f$. Japan. J. Math. (N.S.) {\bf 2} (1976), 191--248.

\bibitem{EY86}
{\sc H. Enomoto, H. Yamada}, The characters of $G_2(2^n)$. Japan. J. Math.
  (N.S.) {\bf12} (1986), 325--377.

\bibitem{GecSU3}
{\sc M. Geck}, Irreducible Brauer characters of the 3-dimensional special
  unitary groups in non-defining characteristic. Comm. Algebra {\bf18} (1990),
  563--584.

\bibitem{GLS}
{\sc D. Gorenstein, R. Lyons, R. Solomon}, \emph{The Classification of the
  Finite Simple Groups. Number~3.} Mathematical Surveys and Monographs,
  American Mathematical Society, Providence, RI, 1998.

\bibitem{Him04}
{\sc F. Himstedt}, Character tables of parabolic subgroups of Steinberg's
  triality groups. J. Algebra {\bf281} (2004), 774--822.

\bibitem{Him07}
{\sc F. Himstedt}, Character tables of parabolic subgroups of Steinberg's
  triality groups $\tw3D_4(2^n)$. J. Algebra {\bf316} (2007), 802--827.

\bibitem{HH09}
{\sc F. Himstedt, S. Huang}, Character table of a Borel subgroup of the
  Ree groups $\tw2F_4(q^2)$. LMS J. Comput. Math. {\bf 12} (2009), 1--53.

\bibitem{HS90}
{\sc G. Hiss, J. Shamash}, $3$-blocks and $3$-modular characters of $G_2(q)$.
  J. Algebra {\bf131} (1990), 371--387.

\bibitem{HS92}
{\sc G. Hiss, J. Shamash}, $2$-blocks and $2$-modular characters of the
  Chevalley groups $G_2(q)$. Math. Comp. {\bf59} (1992), 645--672.

\bibitem{Hw73}
{\sc R.B. Howlett}, On the degrees of Steinberg characters of Chevalley
  groups. Math. Z. {\bf 135} (1973/74), 125--135.

\bibitem{Hum}
{\sc J.E. Humphreys}, Defect groups for finite groups of Lie type.
  Math. Z. {\bf119} (1971), 149--152.

\bibitem{Hup}
{\sc B. Huppert}, \emph{Character Theory of Finite Groups}. Walter de Gruyter
  \& Co., Berlin, 1998.

\bibitem{Ka77}
{\sc D. Kazhdan}, Proof of Springer's hypothesis. Israel J. Math. {\bf28}
  (1977), 272--286.

\bibitem{KM13}
{\sc R. Kessar, G. Malle}, Quasi-isolated blocks and Brauer's height
  zero conjecture. Annals of Math. {\bf178} (2013), 321--384.

\bibitem{Ma90}
{\sc G. Malle}, Die unipotenten Charaktere von $\tw2F_4(q^2)$. Comm. Algebra
  {\bf18} (1990), 2361--2381.

\bibitem{MaH0}
{\sc G. Malle}, Height~0 characters of finite groups of Lie type.
  Represent. Theory {\bf11} (2007), 192--220.

\bibitem{MT}
{\sc G. Malle, D. Testerman}, \emph{Linear Algebraic Groups and Finite
  Groups of Lie Type}. Cambridge Studies in Advanced Mathematics, 133.
  Cambridge University Press, Cambridge, 2011.

\bibitem{Nav}
{\sc G. Navarro}, \emph{Characters and Blocks of Finite Groups}. London Math.
  Soc. Lecture Note Series, 250. Cambridge University Press, Cambridge,
  1998.

\bibitem{NS13}
{\sc G. Navarro, B. Sp\"ath}, On Brauer's height zero conjecture.
  J. Eur. Math. Soc. {\bf16} (2014), 695--747. 

\bibitem{Ol93}
{\sc J.B. Olsson}, \emph{Combinatorics and Representations of Finite Groups}.
  Universit\"at Essen, Fachbereich Mathematik, Essen, 1993.

\bibitem{Suz}
{\sc M. Suzuki}, On a class of doubly transitive groups. Ann. of Math. (2)
  {\bf75} (1962), 105--145.

\bibitem{Wa}
{\sc H.N. Ward}, On Ree's series of simple groups. Trans. Amer. Math. Soc.
  {\bf121} (1966), 62--89.

\bibitem{Wei}
{\sc A.J. Weir}, Sylow $p$-subgroups of the classical groups over finite
  fields with characteristic prime to $p$. Proc. Amer. Math. Soc. {\bf6}
  (1955), 529--533.

\end{thebibliography}
\end{document}